\documentclass[10pt,a4paper]{article}
\usepackage[utf8]{inputenc}
\usepackage[english]{babel}
\usepackage[T1]{fontenc}
\usepackage{arydshln}
\usepackage{amsmath}
\usepackage{amsthm}
\usepackage{amsfonts}
\usepackage{xcolor}
\usepackage{yfonts}
\usepackage{amssymb}
\usepackage{amsthm,bbold}
\usepackage{xcolor}
\usepackage{xfrac}
\usepackage{mathtools}
\usepackage{mdframed}
\usepackage{hyperref}
\hypersetup{hidelinks}
\usepackage{tocloft}
\usepackage{mathrsfs}
\usepackage{geometry}
\usepackage{changepage}
\usepackage{stmaryrd}
\usepackage{tikz-cd}
\usepackage[nameinlink]{cleveref}
\usepackage{comment}

\newmdenv[
  topline=false,
  bottomline=false,
  skipabove=\topsep,
  skipbelow=\topsep
]{siderules}

\theoremstyle{plain}
\newtheorem{theorem}{Theorem}[section]

\newtheorem*{theorem*}{Theorem}
\newtheorem*{theoremfr*}{Théorème}
\newtheorem{definition}[theorem]{Definition}
\newtheorem{proposition}[theorem]{Proposition}
\newtheorem{setup}[theorem]{Setup}

\newtheorem{lemma}[theorem]{Lemma}
\newtheorem{corollary}[theorem]{Corollary}

\theoremstyle{plain}

\theoremstyle{plain}

\theoremstyle{remark}

\usepackage{thmbox}
\usepackage{geometry}
\usepackage{graphicx}
\graphicspath{{images/}}

\newcommand{\pause}{\begin{center}
$\star \hspace{.1cm} \star$
\end{center}
}

\newcommand{\C}{\mathbf{C}}

\newcommand{\bbP}{\mathbf{P}}

\title{Divisorial Mori contractions of submaximal length.}

\author{Bruno Dewer}

\date{}

\begin{document}

\maketitle

\begin{abstract}
A result due to Cho, Miyaoka, Shepherd-Barron \cite{CMSB} and Kebekus \cite{Ke} provides a numerical characterization of projective spaces. More recently, Dedieu and Höring \cite{DH} gave a characterization of smooth quadrics based on similar arguments. As a relative version of \cite{CMSB} and \cite{Ke}, Höring and Novelli proved in \cite{HN} that the locus covered by positive-dimensional fibres in a Mori contraction of maximal length is a projective bundle up to birational modification. We change the length hypothesis and we prove that the exceptional locus of a divisorial Mori contraction of submaximal length is birational either to a projective bundle, or to a quadric bundle.
\end{abstract}

\section{Introduction}

If $X$ is a smooth Fano variety of dimension $n$, its pseudoindex is defined as
$$
l(X) = \min \left\{ -K_X \cdot \Gamma \: | \: \Gamma \subset X \text{ a rational curve} \right\}.
$$
This invariant provides much information on $X$:

\begin{theorem}[\cite{CMSB}, \cite{Ke}, \cite{Mi}, \cite{DH}]
Let $X$ be a smooth Fano variety of dimension $n$. 

\begin{enumerate}
\item[$\bullet$] \cite{CMSB}, \cite{Ke}: If for every rational curve $\Gamma \subset X$, one has $-K_X \cdot \Gamma \geq n+1$, then $X\simeq \bbP^n$. In this case, $l(X) = n+1$ and the equality $-K_X\cdot \Gamma = n+1$ holds iff. $\Gamma$ is a line.

\item[$\bullet$] \cite{Mi}, \cite{DH}: If for every rational curve $\Gamma \subset X$, one has $-K_X \cdot \Gamma \geq n$, then $X$ is isomorphic either to a quadric, or to a projective space. If $l(X) = n$, then $X$ is a quadric and the equality $-K_X \cdot \Gamma = n$ holds iff. $\Gamma$ is a line.
\end{enumerate}
\end{theorem}

Fano manifolds arise in the Mori program as general fibres of Mori fibre spaces. More generally, if $f : X \to Y$ is a \emph{fibration} (i.e., a surjective morphism with connected fibres and $\dim X > \dim Y$) from a smooth variety $X$ such that $-K_X$ is $f$-ample and $F$ is a smooth fibre over a smooth point of $Y$, then $F$ is a Fano manifold.

From now on, we will use the term \emph{contraction} to designate any surjective morphism $f : X \to Y$ with connected fibres such that there exists at least one curve $C \subset X$ contracted by $f$.

\vspace{.2cm}
Let us consider a smooth projective variety $X$ and an elementary Mori contraction $f : X \to Y$. In other words, the curves contracted by $f$ are those curves whose numerical equivalence classes all belong to a single $K_X$-negative extremal ray of the Mori cone of $X$. The length of such a contraction is defined as follows:

\begin{definition}[Length]
\label{def:length}
Let $X$ be a smooth projective variety and $f : X \to Y$ an elementary Mori contraction. The length of $f$ is defined as the minimal degree of $-K_X$ on the rational curves which are contracted by $f$:
$$
l(f) = \min \left\{ -K_X \cdot \Gamma \: | \: \Gamma \subset X \text{ a contracted rational curve} \right\}.
$$
\end{definition}

This notion can naturally be viewed as a relative version of the pseudoindex. Indeed, the pseudoindex of the general positive-dimensional fibre of $f$ can sometimes be recovered from the adjunction formula and $l(f)$. If $f$ is of fibre type (meaning $\dim X > \dim Y$), its general fibre $F$ is a Fano manifold with $K_F = K_X|_F$, yielding $l(F) \geq l(f)$, and $l(F) = l(f)$ if $F$ contains a rational curve $\Gamma$ verifying
\begin{equation}
\label{eq:minimalcurve}
-K_X\cdot \Gamma = l(f).
\end{equation}
If $f$ is birational and its exceptional locus $E$ is a hypersurface of $X$, and $F \subset E$ is a general fibre of $f|_E$ containing a rational curve $\Gamma$ satisfying \color{purple}(\ref{eq:minimalcurve})\color{black}, then $K_E = (K_X+E)|_E$ and $K_F = K_E|_F = (K_X+E)|_F$, so that $l(F)$ and $l(f)$ differ by the degree of the divisor $E$ on contracted curves. Thanks to a classical result due to P.~Ionescu and J.~Wiśniewski, the length of an elementary Mori contraction $X\to Y$ from a smooth variety $X$ is bounded from above:

\begin{theorem}[\cite{Io} Theorem 0.4, \cite{Wi} Theorem 1.1]
\label{thm:IonescuWisniewski}
Assume $f:X\to Y$ is an elementary Mori contraction from a smooth projective variety $X$. Let $E\subset X$ be an irreducible component of the $f$-exceptional locus (or $E = X$ if $f$ is of fibre type) and $F\subset E$ an irreducible component of a fibre. Then
$$
\dim E + \dim F \geq \dim X + l(f) - 1.
$$
\end{theorem}

\noindent
\textbf{Terminology.} The length of $f$ is \emph{maximal} if there exist an irreducible component $E$ of the $f$-exceptional locus (or $E = X$ if $f$ is of fibre type) and $F$ a fibre in $E$ such that $l(f) = \dim E - \dim X + \dim F + 1$. The length of $f$ is \emph{submaximal} if it is not maximal and $f$ admits an irreducible component $E$ of $\mathrm{Exc}(f)$ and a fibre $F \subset E$ such that $l(f) = \dim E - \dim X + \dim F$.

\vspace{.2cm}
We refer to \cite{HN} for a study of elementary Mori contractions of maximal length, both of fibre type and birational. Given $f : X \to Y$ a Mori contraction of maximal length from a smooth projective variety $X$, their main results \cite[Theorems 1.3 \& 1.4]{HN} state that the $f$-exceptional locus (respectively $X$) is birational to a projective bundle over its image if $f$ is birational (respectively, if $f$ is of fibre type).
 
In this paper, we change the assumptions of the above theorem, moving to the case of submaximal length for a birational elementary Mori contraction. We assume moreover that our contraction is \emph{divisorial}, i.e., that the exceptional locus has codimension $1$. Everything is specified in \color{purple}Setup \ref{setup:setup} \color{black}below.

\vspace{.2cm}
\noindent
\textbf{Notation.} If $f : X \to Y$ is a divisorial contraction and $E \subset X$ is the $f$-exceptional locus with $\dim E - \dim f(E) = n$, we denote by $E_\mathrm{eq}\to Z_\mathrm{eq}$ the \emph{equidimensional locus} of $f|_E$, that is, the locus in $E$ which is covered by the $f$-fibres of dimension $n$.

\vspace{.2cm}
\noindent
\textbf{Terminology.} If $f : X \to Y$ is an elementary Mori contraction and $\Gamma \subset X$ is a rational curve contracted by $f$ on which $-K_X$ has the lowest possible degree, in other words
$$
-K_X \cdot \Gamma = l(f),
$$
then we say that $\Gamma$ is an \emph{$f$-minimal curve}, or sometimes just a \emph{minimal curve} when there is no risk of confusion.

\subsection{Setup and statement of the main result}

The following setup contains all the assumptions that the remaining of this paper is based on.

\begin{setup} \label{setup:setup}
Let $f: X \to Y$ be an elementary Mori contraction from a smooth projective variety $X$. We assume that it is birational and divisorial with $E$ its exceptional divisor, which is irreducible by \cite[Proposition 6.10.b]{Deb01}, and we set $Z := f(E)$. If $E_z = f^{-1}(z)$ is a general fibre of $f|_E$ then by the Ionescu-Wiśniewski inequality given in \color{purple}Theorem \ref{thm:IonescuWisniewski}\color{black}, the length of $f$ is bounded from above by $\dim E_z$. We assume here that the length is submaximal, in other words 
$$
l(f) = \dim E_z - 1,
$$
for $E_z$ a fibre of the expected dimension $n := \dim E - \dim Z$. In other words, we have $l(f) = n-1$ (in particular, $n$ must be larger than $1$).
\end{setup}

\pause

We consider an $f$-minimal curve $\Gamma$ in the equidimensional locus $E_\mathrm{eq}$, the existence of which is ensured by \color{purple}Lemma \ref{lem:all fibres}\color{black}, and the Cartier divisor $-E|_E$, which is ample by Kleiman's criterion (see for instance \cite[Theorem 1.27.a]{Deb01}). Then we deduce information on $E_\mathrm{eq}\to Z_\mathrm{eq}$, according to the degree of $-E$ on $\Gamma$. Namely, our main result is the following:

\begin{theorem}
\label{thm:main}
Under the assumptions of \color{purple}Setup \ref{setup:setup}\color{black}, we have either $-E\cdot \Gamma = 1$ or $-E\cdot \Gamma = 2$. Furthermore:

\begin{enumerate}
\item[$\bullet$] If $-E\cdot \Gamma = 2$, then all the fibres of $E_\mathrm{eq} \to Z_\mathrm{eq}$ are normalized by $\bbP^n$ and the fibration $E_\mathrm{eq}\to Z_\mathrm{eq}$ is birational to a family of projective spaces. If moreover $n$ is even, $E_\mathrm{eq}$ is isomorphic to the projectivization of a vector bundle over $Z_\mathrm{eq}$.

\vspace{.2cm}
\item[$\bullet$] If $-E\cdot \Gamma = 1$, then $E_\mathrm{eq}\to Z_\mathrm{eq}$ is birational to a quadric bundle. Moreover, each reducible fibre of $E_\mathrm{eq}\to Z_\mathrm{eq}$ has two irreducible components, the reductions of which are normalized by $\bbP^n$.
\end{enumerate}
\end{theorem}

In \S 4 we provide a case-by-case proof of \color{purple}Theorem \ref{thm:main}\color{black}, according to the degree of the exceptional divisor on the $f$-minimal curves. The proof is split in three parts: the first part is \color{purple}Proposition \ref{lem:generalfibre}\color{black}, in which we show that we have either $-E\cdot \Gamma = 1$ or $-E\cdot \Gamma = 2$. The second part is \color{purple}Theorem \ref{thm:1st half} \color{black}(the case $-E\cdot \Gamma = 2$), and the last part is \color{purple}Theorem \ref{thm:2nd half} \color{black}(the case $-E\cdot \Gamma = 1$). The proof of \color{purple}Proposition \ref{lem:generalfibre} \color{black}relies on a technique which consists in degenerating families of $f$-minimal curves in the general fibre of $E_\mathrm{eq}\to Z_\mathrm{eq}$ to families of $f$-minimal curves contained in singular fibres, and bounding the dimension of said families from below.

\vspace{.2cm}
Restricting to the equidimensional locus $E_\mathrm{eq}\to Z_\mathrm{eq}$ is a necessary condition for \color{purple}Theorem \ref{thm:main}\color{black}. To illustrate this, \S 5 is dedicated to the construction of an example of a non-equidimensional divisorial elementary Mori contraction $X \to Y$, whose exceptional divisor is generically a quadric bundle over its image and contains a special fibre which is isomorphic to a projective space.

\section{Prerequisites}

One of the techniques that we will use to study the fibres contained in the exceptional divisor of a divisorial Mori contraction relies on the degeneration of curves. This requires to introduce the following notion:

\begin{definition}[Deformation family]
Consider a rational curve $\Gamma$ contained in a variety $X$. The datum of this curve is equivalent to the datum of a morphism $\bbP^1 \to X$ birational onto its image, modulo automorphisms of $\bbP^1$. In other words, it corresponds to a point of $\mathrm{RatCurves}^n(X)$: let $\mathcal H$ be an irreducible component of $\mathrm{RatCurves}^n(X)$ containing this point. We say that $\mathcal H$ is a \emph{deformation family of $\Gamma$ in $X$}.
\end{definition}

We refer to \cite[Proposition II.2.11.2]{Kol} for the construction of $\mathrm{RatCurves}^n(X)$ and that of deformations families. By construction, the points of such a deformation family $\mathcal H$ parametrize rational curves, i.e., irreducible and reduced $1$-cycles in $X$. This family $\mathcal H$ is the normalization of a subspace $V \subset \mathrm{Chow}(X)$; let $\overline V$ be the closure of $V$. Then one may consider $\overline{\mathcal H}$ the normalization of $\overline V$, such that $\mathcal H$ is dense in $\overline{\mathcal H}$. By construction, the points of $\overline{\mathcal H}$ parametrize either rational curves in $X$, or $1$-cycles in $X$ that are degenerations of rational curves.

\begin{definition}[Closed deformation family --- unsplit deformation family]
We say that $\overline{\mathcal H}$ is a \emph{closed deformation family} of $\Gamma$ in $X$.

\vspace{.2cm}
A closed deformation family $\overline{\mathcal H}$ is \emph{unsplit} if all its members are reduced and irreducible, in other words if $\mathcal H = \overline{\mathcal H}$.
\end{definition}

Moreover, given $\overline{\mathcal H}$ a closed deformation family of a rational curve in a variety $X$, there exists a \emph{universal family} over $\overline{\mathcal H}$
$$
\mathcal U = \left\{ (p,Z) \in X\times \overline{\mathcal H} \: | \: Z \in \overline{\mathcal H}, \: p\in \mathrm{Supp}(Z)\right\},
$$
with the two projections
\begin{center}
\begin{tikzcd}
\mathcal U \arrow[d,"q",swap] \arrow[rr,"ev"] & & X \\
\overline{\mathcal H}
\end{tikzcd}
\end{center}
In the above, the notation $ev$ stands for ``evaluation morphism''. The existence of $\mathcal U$ is explained in \cite[II.2.8, II.2.12]{Kol}. Moreover, as $\overline{\mathcal H}$ is normal by construction, we have the following:

\begin{lemma}
\label{cor:constant degree on defofamily}
Let $\Gamma$ be a rational curve contained in a variety $X$ and $\overline{\mathcal H}$ a closed deformation family of $\Gamma$ in $X$. Let $L$ be a line bundle on $X$. Then $L$ has the same degree on all the members of $\overline{\mathcal H}$.
\end{lemma}

\begin{proof}
Let $\mathcal U$ be the universal family over $\overline{\mathcal H}$, with $ev : \mathcal U \to X$ the evaluation map. The fibration $q : \mathcal U \to \overline{\mathcal H}$ is equidimensional onto a normal variety, so the line bundle $ev^*L$ has the same degree on all the fibres of $q$ by \cite[I.3.12]{Kol}. But by construction, the fibres of $q$ are the members of the family $\overline{\mathcal H}$.
\end{proof}

This means that each member of $\overline{\mathcal H}$ is a $1$-cycle which is numerically equivalent to a rational curve. It may happen that $\mathcal H \neq \overline{\mathcal H}$, in which case the family $\overline{\mathcal H}$ contains nonreduced or reducible $1$-cycles whose components are supported on rational curves. This may happen in certain situations, such as that of Mori's bend-and-break lemma; see for instance \cite[Proposition 3.5]{Deb01}.

\begin{corollary}
\label{cor:all contracted}
Let $f : X \to Y$ be a morphism, $\Gamma \subset X$ a rational curve contracted by $f$ and $\overline{\mathcal H}$ a closed deformation family of $\Gamma$ in $X$. Then all the members of $\overline{\mathcal H}$ are contracted by $f$.
\end{corollary}


\pause 

When considering the normalization $\nu : X' \to X$ of a variety $X$ with Cartier canonical divisor, we will sometimes need to investigate the divisor $\nu^*K_X - K_{X'}$, which yields some information on the nonnormal locus of $X$:

\begin{lemma}[Conductor]
\label{lem:conductor}
Let $X$ be a variety with Cartier canonical class $K_X$ which satisfies the $S_2$ condition, and $\nu : X' \to X$ it normalization. Then there exists an effective Weil divisor $\mathcal D$ on $X'$ such that $K_{X'}+\mathcal D$ is Cartier, and
$$
\nu^* K_X \simeq K_{X'} + \mathcal D.
$$
Moreover, $\nu(\mathcal D)$ is the nonnormal locus of $X$.

In particular, if $C$ is a projective, irreducible, reduced, Gorenstein curve, then the degree of the divisor $\mathcal D$ involved in its normalization is even.
\end{lemma}

\begin{proof}
A construction of the conductor is provided in \cite[5.2.2]{Kol2}. In general, this is just a subscheme of $X'$, but when $X$ is $S_2$ it is a hypersurface of $X'$. By construction, the support of the conductor is the locus where $\nu$ is not an isomorphism, i.e., the preimage of the nonnormal locus of $X$. The formula
$$
\nu^*K_X \simeq K_{X'} + \mathcal D
$$
follows from \cite[5.7.1]{Kol2}, which holds under the condition that $K_X$ is Cartier and $X$ is $S_2$.

Now consider $C$ a projective reduced and irreducible curve with $K_C$ Cartier, and $C' \to C$ its normalization. By the Riemann-Roch formula for singular curves (see for instance \cite[Exercise IV.1.9]{Har}), we have $\deg K_C = 2p_a(C) - 2$. Hence the degree of the conductor divisor is $\deg K_C - \deg K_{C'} = 2(p_a(C) - g(C))$.
\end{proof}

\pause

Lastly, we need to introduce the following notion about the pullback of the pushforward of a line bundle via a contraction:

\begin{lemma}[Evaluation map]
\label{lem:evaluationmap}
Let $f: X \to Y$ be a contraction with projective fibres and $L$ a line bundle on $X$. Then the $\mathcal O_X$-module $f^*f_*L$ admits a morphism 
$$
f^*f_*L \to L
$$
which we call the \emph{relative evaluation map}, or simply the \emph{evaluation map}, of $L$.
\end{lemma}

\begin{proof}
By \cite[II.5]{Har}, there is a canonical isomorphism of groups
$$
\mathrm{Hom}_{\mathcal O_X}(f^*f_*L,L) \simeq \mathrm{Hom}_{\mathcal O_Y}(f_*L,f_*L),
$$
and the evaluation map $f^*f_*L \to L$ is the element of $\mathrm{Hom}_{\mathcal O_X}(f^*f_*L,L)$ which corresponds to the identity $id : f_*L \to f_*L$.
\end{proof}

Furthermore, if $p\in Y$ is a point such that $f_*L$ is locally free of rank $r$ at $p$, and the fibre $X_p = f^{-1}(p)$ is such that
$$
\dim H^0(X_p,L|_{X_p}) = r,
$$
then the restriction of the evaluation map $f^*f_*L \to L$ to $X_p$ is a morphism of vector bundles
$$
H^0(X_p,L|_{X_p}) \otimes \mathcal O_{X_p} \to L|_{X_p}.
$$
Hence the evaluation map is surjective along the fibre $X_p$ iff. $L|_{X_p}$ is globally generated.

\begin{theorem}[\cite{AW}, Theorem 5.1]
\label{thm:AW}
Let $f : X \to Y$ be a contraction from a projective normal variety with at worst klt singularities. Assume that $f$ is supported by $K_X + rL$ for some line bundle $L$ and $r$ a rational number (in other words, $K_X + rL$ has degree zero on all contracted curves and positive degree on all other curves). Assume that $L$ is $f$-ample, i.e., its restriction to every $f$-fibre is ample.

Consider a fibre $X_p$ of $f$. Assume moreover that
\begin{enumerate}
    \item[$\bullet$] $f$ is a fibration (meaning $\dim X > \dim Y$) and $\dim X_p < r+1$, or
    \item[$\bullet$] $f$ is birational and $\dim X_p \leq r+1$.
\end{enumerate}
Then the evaluation map $f^*f_*L \to L$ is surjective along $X_p$.
\end{theorem}

Moreover, when $f^*f_*L \to L$ is surjective on all of $X$, the following birational modification arises as an extension of the locus where the direct image $f_*L$ is locally free:

\begin{lemma}
\label{lem:localfreemodif}
Let $f:X\to Y$ be a fibration between irreducible quasi-projective varieties and $L$ a line bundle on $X$ such that the evaluation map $f^*f_*L \to L$ is surjective. Let $\mathcal U_L$ be the locus where $f_*L$ is locally free. Then there exists a birational modification $\mu : \mathcal Y \to Y$ which is an isomorphism on $\mathcal U_L$, together with a vector bundle $\mathcal V$ on $\mathcal Y$ which coincides with $f_*L$ on $\mathcal U_L$. 

Moreover, let $\mathcal X$ be the irreducible component of the fibre product $X\times_Y \mathcal Y$ which dominates $X$, with the following commutative diagram.
\begin{center}
\begin{tikzcd}
\mathcal X \arrow[rr,"\mu'"] \arrow[d,"f'",swap] & & X \arrow[d,"f"] \\
\mathcal Y \arrow[rr,"\mu",swap] & & Y
\end{tikzcd}
\end{center}
Then there is a surjective map $f'^*\mathcal V \to \mu'^*L$ which coincides with the evaluation map $f^*f_*L \to L$ on $f^{-1}(\mathcal U_L)$.
\end{lemma}

\begin{proof}
By \cite[Theorem 3.5]{Ro} there exist a birational modification $\mu : \mathcal Y \to Y$ and a vector bundle $\mathcal V$ on $\mathcal Y$ with a surjective map
\begin{equation} \label{eq:kernel}
\mu^*(f_*L) \to \mathcal V
\end{equation}
whose kernel is torsion. On the locus $\mathcal U_L$ where $f_*L$ is locally free, $\mu$ is an isomorphism; in particular, we have $f_*L \simeq \mathcal V$ on $\mathcal U_L$. Moreover, the pullback by $\mu'$ of the surjective map
$$
f^*f_*L \to L
$$
is also surjective. Let $\mathcal K$ denote the kernel of the map \color{purple}(\ref{eq:kernel})\color{black}. Then in the following diagram, the top row is exact, and the vertical arrow is onto
\begin{center}
\begin{tikzcd}[column sep=tiny]
f'^*\mathcal K \arrow[rrrrrr] \arrow[drrrrrrrr,"\alpha",swap]  & & & & & & f'^*\mu^*(f_*L) & \simeq & \mu'^*(f^*f_*L) \arrow[rrr] \arrow[d] & & &  f'^*\mathcal V \arrow[rrr] & & & 0 \\
& & & & & & & & \mu'^*L & & & & & & 
\end{tikzcd}
\end{center}
The map $\alpha$ is zero, since $\mu'^*L$ is a line bundle and $f'^*\mathcal K$ is torsion on the irreducible variety $\mathcal X$. This ensures the existence of a factorization $f'^*\mathcal V \to \mu'^*L$ which is surjective as well.
\end{proof}

\section{The general fibre of the exceptional locus}

In this section, the aim is to identify the general fibre of $f|_E$. It requires that we state first the following result:

\begin{lemma}
\label{lem:all fibres}
Under the assumptions of \color{purple}Setup \ref{setup:setup}\color{black}, every fibre of $E\to Z$ contains an $f$-minimal curve.
\end{lemma}

\begin{proof}
Let $E_z \subset E_\mathrm{eq}$ be a generically reduced fibre of dimension $n$ over a smooth point of $Z_\mathrm{eq}$. Since $E_z$ is locally a complete intersection in $X$, by the adjunction formula we have $-K_{E_z} = -K_E|_{E_z} = (-K_X-E)|_{E_z}$; in particular, $-K_{E_z}$ is ample. Let now $C \subset E_z$ be a rational curve whose degree is minimal with respect to $-K_{E_z}$ among rational curves meeting the smooth locus of $E_z$. Then
\begin{equation} \label{eq:adjunction}
-K_{E_z}\cdot C = -K_X\cdot C -E\cdot C \geq n-1 -E\cdot C.
\end{equation}
By ampleness of $-E|_{E_z}$, it follows that $-K_{E_z}\cdot C \geq n$. Assume by contradiction that $-K_{E_z}\cdot C > n+1$. By \cite[Proposition II.1.3]{Kol}, there exists a deformation family $\mathcal H_z$ of $C$ in $E_z$ such that
$$
\dim \mathcal H_z \geq -K_{E_z}\cdot C + (n-3) > 2n-2,
$$
and considering the universal family over $\mathcal H$,
\begin{center}
\begin{tikzcd}
\mathcal U_z \arrow[rr,"ev"] \arrow[d,"\pi",swap] & & E_z \\
\mathcal H_z
\end{tikzcd}
\end{center}
where $\pi$ is a $\bbP^1$-bundle, we have $\dim \mathcal U_z \geq 2n$. Fixing $p\in E_z$ a general point and considering $\mathcal U_{z,p} = \pi^{-1}(\pi(ev^{-1}(p)))$, we have $\dim \mathcal U_{z,p} \geq n+1$, hence there exists a fibre of $\mathcal U_{z,p} \to E_z$ of positive dimension. In other words, there exists a point $q\in E_z$ and a positive-dimensional family of rational curves through $p$ and $q$ (in particular, meeting the smooth locus of $E_z$), all of which have minimal degree among rational curves meeting $(E_z)_{\mathrm{smooth}}$. By the bend-and-break lemma, there exists such a member which is reducible or nonreduced, which is not possible by the minimality of $-K_{E_z}\cdot C$. As a conclusion, we have eiher $-K_{E_z}\cdot C = n$ or $-K_{E_z}\cdot C = n+1$.

Going back to \color{purple}(\ref{eq:adjunction})\color{black}, we deduce that one of the two following cases occur:
\begin{enumerate}
\item[$\bullet$] $-K_X \cdot C = n-1 = l(f)$, in which case $C$ is an $f$-minimal curve,
\item[$\bullet$] $-K_X \cdot C = n$, in which case $-E\cdot C = 1$ and by \cite[Theorem 7.39.c]{Deb01} we have $(K_X - nE) \sim_f 0$ (in other words $-K_X$ and $-nE$ are relatively linearly equivalent).
\end{enumerate}
Assume by contradiction that the second case holds. In that case, for $\Gamma \subset E$ an $f$-minimal curve, we have $-K_X\cdot \Gamma = n-1$ by our length hypothesis, hence $-E\cdot \Gamma = \frac{n-1}{n}$, which is not an integer. This is a contradiction.

It follows that any generically reduced fibre of dimension $n$ over the smooth locus of $Z$ contains an $f$-minimal curve. This ensures that there exists a deformation family of $f$-minimal curves which dominates $Z$. This family is unsplit by the length hypothesis and \color{purple}Lemma \ref{cor:constant degree on defofamily}\color{black}, hence compact. Since it is proper over $Z$, it is surjective. Hence every fibre of $E \to Z$ contains an $f$-minimal curve.
\end{proof}

\begin{lemma}
\label{lem:covered}
Under the assumptions of \color{purple}Setup \ref{setup:setup}\color{black}, the $f$-minimal curves cover the equidimensional locus $E_\mathrm{eq}$.
\end{lemma}

\begin{proof}
Let $\Gamma \subset E_\mathrm{eq}$ be an $f$-minimal curve. As $X$ is smooth, we may apply \cite[Theorem II.1.3]{Kol}, which ensures the existence of a deformation family $\mathcal H$ of $\Gamma$ in $X$ such that
$$
\dim \mathcal H \geq -K_X \cdot \Gamma - 3 + \dim X = n-4 + \dim X.
$$
Since each member of $\mathcal H$ lies in a fibre of $f$ by \color{purple}Corollary \ref{cor:all contracted}\color{black}, and we have shown in \color{purple}Lemma \ref{lem:all fibres} \color{black}that every fibre of dimension $n$ contains an $f$-minimal curve, there exists such a family $\mathcal H$ which is surjective over $Z$. In particular, it admits a fibration $\mathcal H_\mathrm{eq} \to Z_\mathrm{eq}$, where $\mathcal H_\mathrm{eq}$ is dense in $\mathcal H$, such that the fibre $\mathcal H_z$ over a point $z\in Z_\mathrm{eq}$ parametrizes deformations of $\Gamma$ inside $E_z$. Consider such a fibre $E_z$ of dimension $n = \dim E - \dim Z$; by the above inequality we have
$$
\dim \mathcal H_z = \dim \mathcal H - \dim Z \geq n-4 + \dim X - (\dim X - 1 - n) = 2n-3.
$$
Now let $\mathcal U_z$ be the universal family over $\mathcal H_z$:
\begin{center}
\begin{tikzcd}
\mathcal U_z \arrow[d,"\pi",swap] \arrow[rr,"ev"] & & E_z \\
\mathcal H_z & &
\end{tikzcd}
\end{center}
By the inequality $\dim \mathcal H_z \geq 2n-3$, we have $\dim \mathcal U_z \geq 2n-2$, and we assume by contradiction that $ev$ is not surjective. Hence $ev(\mathcal U_z)$, which is the locus covered in $E_z$ by the minimal curves, has dimension at most $\dim E_z - 1 = n-1$. Hence the general fibre of $ev$ has dimension $n-1$ or more.

Pick a general point $x\in ev(\mathcal U_z)$ and denote $\mathcal H_{z,x} = \pi(ev^{-1}(x))$. This is the space parametrizing the minimal curves through $x$, and it is birational to $ev^{-1}(x)$: indeed, any member of $\mathcal H_{z,x}$ is an $f$-minimal curve $\gamma$ through $x$, and its only preimage by $\pi|_{ev^{-1}(x)}$ is $(x,\gamma)$. Let $\mathcal U_{z,x} = \pi^{-1}(\mathcal H_{z,x})$, then we have
$$
\dim \mathcal U_{z,x} = \dim \mathcal H_{z,x} + 1 \geq n.
$$
Since the image of $ev$ has dimension at most $n-1$, the fibres of $\mathcal U_{z,x} \to E_z$ have dimension at least $1$. Given a point $y\in ev(\mathcal U_z) - \left\{ x \right\}$, the fibre over $y$ inside $\mathcal U_{z,x}$ contains a positive dimensional family of curves through the two fixed points $x$ and $y$. By the bend-and-break lemma, this family admits a reducible member or a nonreduced member, which is not possible since $\mathcal H_z$ is unsplit by \color{purple}Lemma \ref{cor:constant degree on defofamily} \color{black}and the length hypothesis $-K_X\cdot \Gamma = l(f)$.

The conclusion follows that $ev$ is surjective onto $E_z$. This holds for any fibre $E_z$ of dimension $n$, hence for the universal family $\mathcal U$ over $\mathcal H$,
\begin{center}
\begin{tikzcd}
\mathcal U \arrow[rr] \arrow[d] & & E \\
\mathcal H & & 
\end{tikzcd}
\end{center}
the evaluation morphism $\mathcal U \to E$ is surjective onto the equidimensional locus $E_\mathrm{eq}$.
\end{proof}

Now, under the conditions of \color{purple}Setup \ref{setup:setup}\color{black}, we investigate the generically reduced fibres of dimension $n = \dim E - \dim Z$.

\begin{proposition}
\label{lem:generalfibre}
Let $E_z\subset E$ be a generically reduced fibre of dimension $n$ over a smooth point $z$ of $Z$, and let $\Gamma \subset E_z$ be an $f$-minimal curve. Then one of the following cases occurs:
\begin{enumerate}
\item[$(i)$] $-E\cdot \Gamma = 2$ and $E_z \simeq \bbP^n$,
\item[$(ii)$] $-E\cdot \Gamma = 1$.
\end{enumerate}
Moreover, in case $(ii)$, if $E_z$ is irreducible then it is isomorphic to a normal quadric.
\end{proposition}

\begin{proof}
First, consider $F\subset E$ any generically reduced fibre of dimension $n$ over a smooth point of $Z$ and $\Gamma \subset F$ an $f$-minimal curve meeting the smooth locus of $F$. The existence of such a curve is ensured by \color{purple}Lemma \ref{lem:covered}\color{black}.

Since $f(F)$ is a smooth point of $Z$, the fibre $F$ is locally a complete intersection in $X$, and by adjunction we have $K_{F} = K_E|_{F} = (K_X+E)|_{F}$, hence $-K_F$ is ample. Moreover, since $\Gamma$ meets the smooth locus of $F$, by \cite[Theorem II.1.3]{Kol} we have the existence of a deformation family $\mathcal H$ of $\Gamma$ in $F$ such that
$$
\dim \mathcal H  \geq 2n-4 - E\cdot \Gamma.
$$
But $\mathcal H$ is unsplit by the length hypothesis and \color{purple}Lemma \ref{cor:constant degree on defofamily}\color{black}, so we know thanks to \cite[Theorem 0.1]{CMSB} that $\dim \mathcal H$ is bounded from above by $2n-2$, and that this bound is reached if and only if $F$ is normalized by $\bbP^n$. Indeed, in that case $\mathcal H$ induces a complete family of rational curves over the normalization $F'$ of $F$ which have minimal degree with respect to the pullback of $-K_F$, and this implies $F'\simeq \bbP^n$ with $\Gamma \subset F$ the image of a line. In particular, $-E\cdot \Gamma$ can only equal $1$ or $2$.

$\bullet$ Let's examine the situation $(i)$ $-E\cdot \Gamma = 2$. By semicontinuity, for any generically reduced fibre $E_z$ of dimension $n$ over a smooth point of $Z$ and $\Gamma \subset E_z$ a minimal curve meeting the normal locus of $E_z$, the deformations of $\Gamma$ inside $E_z$ form at least one family of dimension $2n-2$, and therefore, by \cite[Theorem 0.1]{CMSB}, $E_z$ is normalized by $\bbP^n$ with $\Gamma$ the image of a line $\ell$:
$$
\nu : \bbP^n \to E_z.
$$
As $E_z$ is a local complete intersection in the smooth variety $X$, it is Cohen-Macaulay and Gorenstein. By \color{purple}Lemma \ref{lem:conductor} \color{black}the conductor divisor $\mathcal D$ of this normalization is such that
$$
\nu^* K_{E_z} \simeq K_{\bbP^n} + \mathcal D.
$$
By adjunction, we have
$$
-K_{E_z}\cdot \Gamma = (-K_X-E)\cdot \Gamma = n+1.
$$
Hence we have $-K_{E_z}\cdot \Gamma = n+1 = -K_{\bbP^n}\cdot \ell$. As numerical equivalence implies linear equivalence on $\bbP^n$, the following equality of Cartier divisors holds:
$$
\nu^*K_{E_z} = K_{\bbP^n}.
$$
So the conductor $\mathcal D$ is trivial. Its support is the nonnormal locus of $E_z$ (this is an application of \color{purple}Lemma \ref{lem:conductor}\color{black}, as $E_z$ is Cohen-Macaulay and Gorenstein). In other words $E_z$ is normal in codimension $1$, and since it is Cohen-Macaulay we have $E_z \simeq \bbP^n$.

$\bullet$ Now let's see what happens in the situation $(ii)$ $-E\cdot \Gamma = 1$, under the hypothesis that $E_z$ is an irreducible fibre of dimension $n$ above a smooth point of $Z$. We pick $\Gamma \subset E_z$ a minimal curve meeting the normal locus of $E_z$ and we denote by $\mathcal H_z$ a deformation family of $\Gamma$ inside $E_z$. By semicontinuity, the dimension of $\mathcal H_z$ can be equal to $2n-2$ or $2n-3$. Assume by contradiction that $\dim \mathcal H_z = 2n-2$. In particular, by \cite[Theorem 0.1]{CMSB} the normalization of $E_z$ is isomorphic to $\bbP^n$, with $\Gamma$ the image of a line $\ell$:
$$
\nu : \bbP^n \to E_z.
$$
From \color{purple}Lemma \ref{lem:conductor}\color{black}, the conductor divisor $\mathcal D$ is such that
$$
\nu^*K_{E_z} \simeq K_{\bbP^n} + \mathcal D.
$$
Moreover, since $E_z$ is locally a complete intersection, it is Gorenstein, and by adjunction we have
$$
-K_{E_z}\cdot \Gamma = (-K_X-E)\cdot \Gamma = n = -K_{\bbP^n}\cdot \ell - 1.
$$
Hence the conductor divisor $\mathcal D$ on $\bbP^n$ satisfies $\mathcal D\cdot \ell = 1$, which tells us that $\mathcal D$ is a hyperplane of $\bbP^n$. In addition $\nu^*(-E|_{E_z})$ is also a hyperplane for the same reason.

The divisor $-E|_{E_z}$ being ample, we may consider an irreducible curve $\gamma \subset E_z$ given as a complete intersection of general divisors of the linear system $|-dE|_{E_z}|$ for $d\gg 0$ an odd integer. This curve meets transversally the image of $\mathcal D$ and if $\gamma' \subset \bbP^n$ stands for its proper transform, the restriction $\gamma' \to \gamma$ is the normalization of $\gamma$. By the adjunction formula 
$$
\begin{array}{rcl}
K_{\gamma'} + \mathcal D|_{\gamma'} & = & (K_{\bbP^n} - d(n-1)\nu^*E + \mathcal D)|_{\gamma'} \\
& = & \nu^*(K_{E_z} - d(n-1)E)|_{\gamma'} \\
& = & (\nu|_{\gamma'})^*K_\gamma,
\end{array}
$$
in other words the conductor divisor of $\gamma' \to \gamma$ is $\mathcal D|_{\gamma'}$. Since $\mathcal D$ and $\nu^*(-E|_{E_z})$ are hyperplane divisors on $\bbP^n$, The divisor $\mathcal D|_{\gamma'}$ has degree $d^{n-1}$, which is odd. This is not possible, since the conductor divisor of a curve always has an even degree, by \color{purple}Lemma \ref{lem:conductor}\color{black}.

So, under the hypothesis that $E_z$ is irreducible and $-E\cdot \Gamma = 1$, we have $\dim \mathcal H_z = 2n-3$ for every deformation family $\mathcal H_z$ of $\Gamma$ inside $E_z$. It remains to be proven that $E_z$ is a normal quadric. Since $K_{E_z}\cdot \Gamma = nE\cdot \Gamma$, the restriction $nE|_{E_z}$ is linearly equivalent to $K_{E_z}$, as a consequence of \cite[Theorem 7.39.c]{Deb01}. Moreover $-E|_{E_z}$ is ample and the normalization $\nu : E_z' \to E_z$ is such that
$$
\begin{array}{rcl}
(K_{E_z'}-n\nu^*E)\cdot (-\nu^*E)^{n-1} & = & (K_{E_z'}-\nu^*K_{E_z})\cdot (-\nu^*E)^{n-1} \\
 & = & -\mathcal D \cdot (-\nu^*E)^{n-1}.
\end{array}
$$
Assume by contradiction that $E_z$ is nonnormal. We have $\mathcal D >0$, or else $E_z$ would be regular in codimension one and therefore normal, since it is Cohen-Macaulay. This yields
$$
(K_{E_z'}-n\nu^*E)\cdot (-\nu^*E)^{n-1} < 0,
$$
in which case there exists a birational morphism $E_z' \to \bbP^n$ such that $-E|_{E_z}$ is the pullback of the hyperplane polarization, by \cite[Proposition 2.13]{Hö}. This is not possible, as $\dim \mathcal H_z = 2n-3$, whereas the unique family of minimal curves (i.e., lines) in $\bbP^n$ has dimension $2n-2$. We may thus conclude that $E_z$ is normal, with the equality of ample Cartier divisors $-K_{E_z} = -nE|_{E_z}$. By \cite[Theorem 3.1.6]{BS}, the generalization of a result by Kobayashi and Ochiai, we may conclude that $E_z$ is isomorphic to a quadric. 
\end{proof}

\section{Main theorem}

\subsection{The case $-E\cdot \Gamma = 2$}

Under the conditions of \color{purple}Setup \ref{setup:setup}\color{black}, we assume that we have $-E\cdot \Gamma = 2$ for every $f$-minimal curve $\Gamma$. The goal is to exhibit a projective bundle as a birational model for the locus covered by the $n$-dimensional fibres. We recall the notation 
$$
E_\mathrm{eq} \to Z_\mathrm{eq}
$$
for the equidimensional locus.

\begin{theorem}
\label{thm:1st half}
In \color{purple}Setup \ref{setup:setup} \color{black}and under the condition $-E\cdot \Gamma = 2$, all the fibres of $E_\mathrm{eq} \to Z_\mathrm{eq}$ are normalized by $\bbP^n$. The normalization $E'$ of $E_\mathrm{eq}$ is a family of projective spaces over the normalization $Z'$ of $Z_\mathrm{eq}$ which is locally trivial for the analytic topology, such that the following diagram is commutative:
\begin{center}
\begin{tikzcd}
E' \arrow[rr] \arrow[d] & & E_\mathrm{eq} \arrow[d] \\
Z' \arrow[rr] & & Z_\mathrm{eq}
\end{tikzcd}
\end{center}
If moreover $n$ is even, $E_\mathrm{eq}$ is isomorphic to the projectivization of a rank $n+1$ vector bundle over $Z_\mathrm{eq}$.
\end{theorem}

The proof is broken down into the following \color{purple}Lemmas \ref{lem:projspacedegen}\color{black}, \color{purple}\ref{lem:projbundle} \color{black}and \color{purple}\ref{lem:evenprojbundle}\color{black}.

\begin{lemma}
\label{lem:projspacedegen}
Let $E \to Z$ be an equidimensional fibration onto a normal variety with $n = \dim E - \dim Z$, such that the general fibre is a projective space. 

For such a general fibre, let $\ell$ be a line inside it. If there exists a relatively ample line bundle $L$ on $E$ such that for any contracted rational curve $\Gamma$ we have
$$
L\cdot \Gamma \geq d := L\cdot \ell,
$$
then all the fibres are irreducible and generically reduced. Moreover, there exists a finite and birational morphism $E' \to E$ where $E'$ is a family of projective spaces over $Z$. If $E$ is normal, then we have $E' \simeq E$.
\end{lemma}

\begin{proof}
The proof that all fibres are irreducible and generically reduced is essentially the same as that of \cite[Proposition 3.1]{HN}, but without the assumption that $E$ is normal. The argument is the following: we assume by contradiction that there exists a fibre $E_z$ which is either reducible or not generically reduced:
$$
E_z = m_1D_1 + \cdots + m_sD_s
$$
where all the $m_i'$s are nonzero. Without loss of generality we may assume $m_1 \geq 2$ or $s \geq 2$. Either way, by \cite[I.3.12]{Kol} we have
\begin{equation}
\label{eq:ineq}
d^n = (L|_{E_z})^n > (L|_{D_1})^n.
\end{equation}
Now consider $C\subset E_z$ a $1$-cycle obtained as a degeneration of lines in the general fibre, so that $L\cdot C = d$. We deduce that $C$ is irreducible and reduced from the length condition on the relatively ample line bundle $L$, namely $L\cdot \Gamma \geq d$ for any reduced and irreducible component $\Gamma$ of $C$. We may assume without loss of generality that the curve $C$ lies in $D_1$. By semicontinuity, a deformation family of $C$ inside $D_1$ has dimension $2n-2$ or more, and \cite[Theorem 0.1]{CMSB} ensures that $D_1$ is normalized by a projective space, with $C$ the image of a line. Since $L\cdot C = d$, we obtain that the pullback of $L$ by the normalization morphism is isomorphic to $\mathcal O_{\bbP^n}(d)$, and thus
$$
(L|_{D_1})^n = d^n
$$
which is a contradiction to the inequality \color{purple}(\ref{eq:ineq})\color{black}.

As a result, all the fibres are generically reduced and irreducible. If $E_z$ is a fibre, and $C\subset E_z$ is a curve obtained as a degeneration of lines, we have $L\cdot C = d$ and a deformation family of $C$ in $E_z$ has dimension $2n-2$ or more. By \cite[Theorem 0.1]{CMSB} again, the normalization of $E_z$ is isomorphic to $\bbP^n$, with $C$ the image of a line :
$$
\nu : \bbP^n \to E_z.
$$
Moreover, the polarization $L$ has degree $d$ on the lines, so $\nu^*L \simeq \mathcal O_{\bbP^n}(d)$.

Now we consider the simultaneous normalization of $E\to Z$ whose existence is given by \cite[Theorem 12]{Kol3}. This is a finite birational modification $E' \to E$ such that all the fibres of $E' \to Z$ are normal; in this particular case, they are isomorphic to $\bbP^n$. Indeed, any fibre $E_z$ of $E \to Z$ is normalized by $\bbP^n$, and if $E'_z$ denotes the fibre of $E' \to Z$ over the same point, then the finite a birational morphism $E'_z \to E_z$ factors through
$$
E'_z \to \bbP^n
$$
since $E'_z$ is normal. It follows that $E'_z \simeq \bbP^n$ since a finite and birational morphism onto a normal variety is an isomorphism by Zariski's main theorem; see \cite{Zar}.

If $E$ is normal, then $E' \to E$ is an isomorphism as a finite and birational morphism onto a normal variety.
\end{proof}

Under the conditions of \color{purple}Setup \ref{setup:setup}\color{black}, and under the condition $-E\cdot \Gamma = 2$, we want to apply \color{purple}Lemma \ref{lem:projspacedegen} \color{black}to $E_\mathrm{eq} \to Z_\mathrm{eq}$. If $Z_\mathrm{eq}$ is normal, then \color{purple}Lemma \ref{lem:projspacedegen} \color{black}ensures that the normalization of $E_\mathrm{eq}$ is a family of projective spaces over $Z_\mathrm{eq}$. If $Z_\mathrm{eq}$ is not normal, let us consider its normalization:

\begin{lemma}
\label{lem:projbundle}
Let $Z'$ be the normalization of $Z_\mathrm{eq}$ and $E'$ the normalization of $E_\mathrm{eq}$. Then the fibration $E'\to Z'$ is a family of projective spaces which is locally trivial for the analytic topology.

In addition, the normalization of any fibre of $E_\mathrm{eq}\to Z_\mathrm{eq}$ is isomorphic to $\bbP^n$.
\end{lemma}

\begin{proof}
We know from \color{purple}Proposition \ref{lem:generalfibre} \color{black}that the general fibre of $E' \to Z'$ is a projective space. Moreover, the pullback of the divisor $-E|_{E_\mathrm{eq}}$ to $E'$ satisfies the hypothesis of \color{purple}Lemma \ref{lem:projspacedegen}\color{black}. Indeed, by \color{purple}Proposition \ref{lem:generalfibre}\color{black}, the minimal degree $-E|_E$ on contracted rational curves is reached on the lines on the fibres which are projective spaces. Hence all the fibres of $E' \to Z'$ are projective spaces by \color{purple}Lemma \ref{lem:projspacedegen}\color{black}, since $E'$ is normal. By \cite{FG}, this family of projective spaces over $Z'$ is relatively trivial in the analytic topology. 

In this situation, the $f$-minimal curves in $E_\mathrm{eq}$ are the images of lines. Let us denote by $\lambda$ the normalization morphism $E' \to E_\mathrm{eq}$. The variety $Z'$ being normal, by \cite[I.3.12]{Kol} the degree of $\lambda^*(-E|_{E_\mathrm{eq}})$ on each line contained in a fibre of $E' \to Z'$ is equal to $2$.

Now we consider $E_z\subset E_\mathrm{eq}$ the fibre over a point $z\in Z_\mathrm{eq}$. Let us denote by $\nu : Z' \to Z_\mathrm{eq}$ the normalization of $Z_\mathrm{eq}$ and pick a point $z' \in \nu^{-1}(z)$. The fibre in $E'$ over $z'$ is isomorphic to $\bbP^n$, and since $\lambda : E'\to E_\mathrm{eq}$ is a finite morphism, it yields a finite morphism $\bbP^n \to E_z$. In this situation, the minimal curves which cover $E_z$ (see \color{purple}Lemma \ref{lem:covered}\color{black}) are the images of the lines of $\bbP^n$.

Since $-E\cdot \Gamma = 2$ and $\lambda^*(-E_{E_\mathrm{eq}})\cdot \ell = 2$ for $\ell$ a line contained in any fibre of $E' \to Z'$, and the restriction of $\lambda$ to $\ell$ is a finite morphism, say of degree $\delta$
$$
\lambda|_\ell : \ell \xrightarrow[]{\delta:1} \lambda(\ell),
$$
for $\ell \subset \bbP^n$ a general line and $\Gamma = \lambda(\ell) \subset E_z$ a minimal curve we have
$$
2 = -E\cdot \Gamma = \frac{1}{\delta}\lambda^*(-E)\cdot \ell = \frac{2}{\delta}
$$
therefore $\delta=1$. The morphism $\bbP^n \to E_z$ is thus birational along the general line $\ell \subset \bbP^n$, so it is birational. Since it is finite, for $(E_z)'$ the normalization of $E_z$ we have a factorization $\bbP^n \to (E_z)'$ which is birational and finite, hence an isomorphism.
\end{proof}

If $n$ is even, we may consider the Cartier divisor 
$$
J = -K_X +\left( \frac{n}{2}-1\right) E
$$
which has degree $1$ on the $f$-minimal curves. By \cite[Theorem 7.39.c]{Deb01}, the divisor $E+2J$ is relatively trivial, so $J\sim_f -\frac{1}{2}E$ is relatively ample.

\begin{lemma}
\label{lem:evenprojbundle}
If $n$ is even, the fibration $E_\mathrm{eq} \to Z_\mathrm{eq}$ is a projective bundle, in other words $E_\mathrm{eq}$ is isomorphic to the projectivization of a vector bundle over $Z_\mathrm{eq}$.
\end{lemma}

\begin{proof}
The relatively ample Cartier divisor 
$$
J = -K_X +\left( \frac{n}{2}-1\right) E
$$
has degree $1$ on the $f$-minimal curves, and
$$
K_X + (n-1)J
$$
is relatively trivial by \cite[Theorem 7.39.c]{Deb01}. Thanks to \color{purple}Theorem \ref{thm:AW}\color{black}, if $E_z \subset E_\mathrm{eq}$ is any fibre, the evaluation map $f^*f_*J \to J$ is surjective along $E_z$.

Now let $z\in Z_\mathrm{eq}$ be the image of $E_z$ and $U$ an affine neighbourhood of $z$ in $Y$. The direct image $(f_*J)|_U$ being a coherent sheaf on an affine variety, it is endowed with a surjection of the following form:
$$
\mathcal O_U^{\oplus m+1} \to f_*J|_U.
$$
Then we may take the pullback of this by $f$ and its restriction to $E_z$, yielding a surjection $\mathcal O_{E_z}^{\oplus m+1} \to (f^*f_*J)|_{E_z}$. Since $(f^*f_*J)|_{E_z} \to J|_{E_z}$ is surjective as well, the composition is onto, namely
$$
\mathcal O_{E_z}^{\oplus m+1} \to J|_{E_z}.
$$
Now let $e_0,...,e_m$ be the vectors of the canonical basis of $H^0(E_z,\mathcal O_{E_z}^{\oplus m+1})$ and $\sigma_0,...,\sigma_m$ their images in $H^0(E_z,J|_{E_z})$ ($1 \leq i \leq m+1$). Then the $\sigma_i$'s do not vanish simultaneously, and they induce a morphism
$$
\theta : E_z \to \bbP^m
$$
such that $\theta^*\mathcal O_{\bbP^m}(1) = J|_{E_z}$.

From \color{purple}Lemma \ref{lem:projbundle} \color{black}we know that $E_z$ is normalized by $\bbP^n$. Let us denote $\lambda : \bbP^n \to E_z$ the normalization morphism, and $J' = \lambda^*(J|_{E_z})$. Then $J'$ is a line bundle on $\bbP^n$ which has degree $1$ on the lines, yielding $J' \simeq \mathcal O_{\bbP^n}(1)$. The composition $\theta \circ \lambda$ is a morphism
$$
\bbP^n \to \bbP^m
$$
such that $\lambda^*\theta^*\mathcal O_{\bbP^m}(1) = J' \simeq \mathcal O_{\bbP^n}(1)$. Since $J'$ is ample and $(J')^n = 1$, $\theta \circ \lambda$ is a finite morphism onto a linear subspace of dimension $n$, hence an isomorphism.

From the above, we deduce that $\theta : E_z \to \bbP^n$ is birational and finite, hence an isomorphism. The fibration $E_\mathrm{eq}\to Z_\mathrm{eq}$ is thus a family of projective spaces. It is endowed with a relatively ample polarization $J$ whose restriction to any fibre $E_z$ is $\mathcal O_{\bbP^n}(1)$ via the isomorphism $E_z \simeq \bbP^n$, and the existence of this relative hyperplane polarization ensures $E_\mathrm{eq} \simeq \bbP_{Z_\mathrm{eq}}(f_*J)$.
\end{proof}

\subsection{The case $-E\cdot \Gamma = 1$}

Under the conditions of \color{purple}Setup \ref{setup:setup}\color{black}, we assume now $-E\cdot \Gamma = 1$ for $\Gamma$ any $f$-minimal curve. We aim to construct a quadric bundle as a birational model for $E_\mathrm{eq}\to Z_\mathrm{eq}$. Let us denote $L := \mathcal O_{E_\mathrm{eq}}(-E|_{E_\mathrm{eq}})$.

\begin{theorem}
\label{thm:2nd half}
In \color{purple}Setup \ref{setup:setup} \color{black}and under the condition $-E\cdot \Gamma = 1$, there exist:

\begin{enumerate}
\item[$\bullet$] birational morphisms $E' \to E_\mathrm{eq}$ and $Z' \to Z_\mathrm{eq}$ with $E'$ and $Z'$ normal,

\item[$\bullet$] a rank $n+2$ vector bundle $\mathcal E$ over $Z'$,

\item[$\bullet$] and a quadric bundle $\mathcal Q \subset \bbP_{Z'}(\mathcal E)$,
\end{enumerate}
such that $E'$ is the normalization of $\mathcal Q$ and the following diagram is commutative:
\begin{center}
\begin{tikzcd}
\mathcal Q \arrow[drr] & & E' \arrow[ll] \arrow[rr] \arrow[d] & & E_\mathrm{eq} \arrow[d] \\
& & Z' \arrow[rr] & & Z_\mathrm{eq}
\end{tikzcd}
\end{center}
Moreover, each reducible fibre of $E_\mathrm{eq}\to Z_\mathrm{eq}$ has two irreducible components, the reductions of which are normalized by $\bbP^n$. Each irreducible and generically reduced fibre of $E' \to Z'$ is isomorphic to a quadric, and the reduction of any nonreduced fibre of $E' \to Z'$ is isomorphic to $\bbP^n$.
\end{theorem}

The proof is broken down into \color{purple}Lemma \ref{lem:quadricbundle}\color{black}, \color{purple}Corollary \ref{cor:modifiedfibres} \color{black}and \color{purple}Corollary \ref{cor:originalfibres}\color{black}.

\begin{lemma}
\label{lem:quadricbundle}
There exist birational morphisms $Z' \to Z_\mathrm{eq}$, $E'\to E_\mathrm{eq}$ with $E'$ and $Z'$ normal, a fibration $f' : E' \to Z'$ and a quadric bundle $\mathcal Q$ over $Z'$, such that $E'$ is the normalization of $\mathcal Q$. Furthermore, the normalization $E' \to \mathcal Q$ is birational along the reduction of each irreducible component of any $f'$-fibre. The reduction of each irreducible component of any reducible or nonreduced fibre of $f'$ is isomorphic to $\bbP^n$. 

In addition, if $\Gamma$ is an $f$-minimal curve in a fibre of $E_\mathrm{eq}\to Z_\mathrm{eq}$, and $\Gamma'$ is any irreducible component of its preimage in $E'$, then the image of $\Gamma'$ in $\mathcal Q$ is a line.
\end{lemma}

\begin{proof}
The general fibre of $E_\mathrm{eq} \to Z_\mathrm{eq}$ is a normal quadric by \color{purple}Proposition \ref{lem:generalfibre}\color{black}.

We consider the birational modification $\mu : \mathcal Z \to Z_\mathrm{eq}$ given by \color{purple}Lemma \ref{lem:localfreemodif}\color{black}, and the vector bundle $\mathcal V$ on $\mathcal Z$ which coincides with $\mu^*f_*L$ over the locus where $f_*L$ is locally free.

On the component $\mathcal F$ of the fibre product $E_\mathrm{eq}\times_{Z_\mathrm{eq}}\mathcal Z$ which dominates $E_\mathrm{eq}$, in the commutative square
\begin{center}
\begin{tikzcd}
\mathcal F \arrow[rr,"\mu'"] \arrow[d,"g",swap] & & E_\mathrm{eq} \arrow[d,"f"] \\
\mathcal Z \arrow[rr,"\mu",swap] & & Z_\mathrm{eq}
\end{tikzcd}
\end{center}
there is a surjective map of vector bundles
$$
g^*\mathcal V \to \mu'^*L
$$
by \color{purple}Lemma \ref{lem:localfreemodif}\color{black}. Now let $\eta : Z' \to  \mathcal Z$ and $\eta' : E' \to \mathcal F$ be the normalizations, then the pullback of the above map $g^*\mathcal V \to \mu'^*L$ by $\eta'$ is surjective:
$$
\eta'^*g^*\mathcal V \to \eta'^*\mu'^*L.
$$
In the following commutative diagram
\begin{center}
\begin{tikzcd}
E' \arrow[rr,"\eta'"] \arrow[d,"f'",swap] & & \mathcal F \arrow[d,"g"] \\
Z' \arrow[rr,"\eta",swap] & & \mathcal Z
\end{tikzcd}
\end{center}
we have $\eta'^*g^*\mathcal V \simeq f'^*\eta^*\mathcal V$. Hence there is a surjective map
$$
f'^*\eta^*\mathcal V \to \eta'^*\mu'^*L
$$
which induces a factorization of $f'$ by the universal property of projectivized bundles
\begin{center}
\begin{tikzcd}
E' \arrow[rr,"\chi"] \arrow[dr,"f'",swap] & & \bbP_{Z'}(\eta^*\mathcal V) \arrow[dl] \\
& Z' &
\end{tikzcd}
\end{center}
We denote the image of $\chi$ by $\mathcal Q$. By \cite[I.3.12]{Kol} every fibre of $\mathcal Q \to Z'$ is a quadric, in other words $\mathcal Q$ is a quadric bundle. Let us denote $\mathcal J = \eta'^*\mu'^*L$, then we have $\chi^*\mathcal O_\mathcal{Q} (1) = \mathcal J$.

Let $E'_w\subset E'$ be an irreducible fibre above a point $w\in Z'$. For a point $w$ such that $\eta(w)$ is a smooth point of $\mathcal Z$, $E'_w$ is isomorphic to its image in $E_\mathrm{eq}$ via $\mu'\circ \eta'$, which is an irreducible quadric by \color{purple}Proposition \ref{lem:generalfibre}\color{black}. Under the additional condition that $\eta(w)$ is outside the $\mu$-exceptional locus, $\mu'\circ \eta'$ is an isomorphism around $E'_w$, and we have $\chi(E'_w)\simeq E'_w$. Therefore $\chi$ is birational, and the general fibre $E'_w$ of $E'\to Z'$ is a quadric with $\mathcal J|_{E'_w}$ the hyperplane polarization.

For any $w\in Z'$, we set $z=\mu\circ \eta(w)$ and $E_z=\mu'\circ \eta'(E'_w)$. Then $E_z$ is the fibre in $E_\mathrm{eq}$ over $z\in Z_\mathrm{eq}$. By construction, $\mu'$ is an isomorphism onto $E_z$, and $\eta'$ is finite, so the morphism $E'_w \to E_z$ is finite, ensuring that $\mathcal J =\eta'^*\mu'^*L$ is ample on $E'_w$. This ensures that $\chi$ is finite, and since it is birational, the conclusion follows that $E'$ is the normalization of $\mathcal Q$.

Now let $\Gamma' \subset E'_w$ be a curve obtained as a degeneration of lines in the smooth fibres of $E_\mathrm{eq}\to Z_\mathrm{eq}$. By \cite[I.3.12]{Kol} we get $\mathcal J\cdot \Gamma' = 1$. In particular, its image $\Gamma = \mu'\circ \eta'(\Gamma')$ in $E_z$ is an $f$-minimal curve. Let $\ell = \chi(\Gamma')$ denote its image in the quadric $\mathcal Q_w = \chi(E'_w)$. If the degree of the finite map 
$$
\chi|_{\Gamma'} : \Gamma' \to \mathcal \ell \subset Q_w
$$
is equal to $\delta$, then we have
$$
1 = \mathcal J \cdot \Gamma' =\delta \mathcal O_{\mathcal Q_w}(1)\cdot \ell \geq \delta
$$
hence $\delta = 1$, and $\ell$ is a line. Since $\mathcal J = \eta'^*\mu'^*L$ and $\mathcal J$ has degree $1$ on the lines of the general fibre of $E' \to Z'$, any $f$-minimal curve $\Gamma \subset E_z$ is the image of $\Gamma' \subset E'$ a degeneration of lines, and $\chi(\Gamma')$ is a line in $\mathcal Q_w$.

If $E'_w$ is irreducible and generically reduced, then $\mathcal Q_w = \chi(E'_w)$ is irreducible as well. Since $\mathcal J = \chi^*\mathcal O_{\mathcal Q}(1)$ with
$$
(\mathcal J|_{E'_w})^n = 2 = (\mathcal O_{\mathcal Q_w}(1))^n
$$
the finite morphism $E'_w \to \mathcal Q_w$ is birational.

Now assume $E'_w$ is reducible or not generically reduced, and $D$ is the reduction of any of its irreducible components. In that case the ample polarization $\mathcal J|_{E'_w}$ verifies
$$
(\mathcal J|_{D})^n < (\mathcal J|_{E'_w})^n = 2
$$
so $(\mathcal J|_{D})^n = 1$. The Cartier divisor $\mathcal J|_{E'_w}$ is basepoint-free since it is the pullback via the morphism $\chi$ of $\mathcal O_{\mathcal Q_w}(1)$, which is basepoint-free. It follows that the morphism $D \to \chi(D)$ induced by the linear system $|\mathcal J|_D|$ is birational and finite onto its image. In particular, $\chi(D)$ is isomorphic to $\bbP^n$, which is normal, so we have $D \simeq \bbP^n$ by Zariski's main theorem.
\end{proof}

\begin{corollary}
\label{cor:modifiedfibres}
If $E'_w$ is a fibre of $E' \to Z'$, then it is either isomorphic to a normal quadric, or as a cycle it is one of the following:
\begin{enumerate}
\item[$\bullet$] reducible and reduced, namely: $E'_w = D_1 + D_2$ with $D_1 \simeq D_2 \simeq \bbP^n$,
\item[$\bullet$] nonreduced, namely: $E'_w = 2D$ with $D \simeq \bbP^n$.
\end{enumerate}
\end{corollary}

\begin{proof}
Given a fibre $\mathcal Q_w$ of the quadric bundle $\mathcal Q \to Z'$ and $E'_w = \chi^{-1}(\mathcal Q_w)$, there are three possibilities:

\begin{enumerate}
\item[$\bullet$] The quadric $\mathcal Q_w$ is normal, in which case for $D$ the reduction of an irreducible component of $E'_w$ which dominates $\mathcal Q_w$, the morphism $D\to \mathcal Q_w$ is birational and finite by \color{purple}Lemma \ref{lem:quadricbundle}\color{black}. Since $\mathcal J$ is $f'$-ample, by \cite[I.3.12]{Kol} we have
$$
2 = (\mathcal J|_{E'_w})^n \geq (\mathcal J|_D)^n = \mathcal O_{\mathcal Q_w}(1)^n = 2,
$$
so $E'_w = D$ is irreducible and generically reduced, and the morphism $E'_w \to \mathcal Q_w$ is birational and finite by \color{purple}Lemma \ref{lem:quadricbundle}\color{black}, hence an isomorphism.
\item[$\bullet$] $\mathcal Q_w$ is reducible, in which case we have the equality of $n$-cycles $\mathcal Q_w = Q_1+Q_2$ where $Q_i \simeq \bbP^n$. In this case $D_i = \chi^{-1}(Q_i)$ for $i=1,2$ are the reductions of two irreducible components of $E'_w$, and since the $Q_i$ are normal and $D_i \to Q_i$ is birational and finite by \color{purple}Lemma \ref{lem:quadricbundle} \color{black}we have $D_i \simeq Q_i$. Moreover, $\mathcal J|_{D_i}\simeq \mathcal O_{\bbP^n}(1)$ via the isomorphism $D_i \simeq \bbP^n$, and we have the equality of cycles
$$
E'_w = D_1 + D_2
$$
since $(\mathcal J|_{E'_w})^n = 2$ by \cite[I.3.12]{Kol} and $(\mathcal J|_{D_i})^n = 1$ for $i = 1,2$. 

\item[$\bullet$] $\mathcal Q_w$ is nonreduced, in which case $\mathcal Q_w = 2P$ with $P\simeq \bbP^n$. By Zariski's main theorem, the reduction $D$ of any irreducible component of $E'_w$ is isomorphic to $P$ since $P$ is normal and $D \to P$ is birational and finite (\color{purple}Lemma \ref{lem:quadricbundle}\color{black}). Moreover $\mathcal J|_{D} \simeq \mathcal O_{\bbP^n}(1)$ via the isomorphism $D \simeq \bbP^n$ since $\mathcal J$ has degree $1$ on the lines of $D$. By \cite[I.3.12]{Kol} we have $(\mathcal J|_{E'_w})^n = 2$, whereas $(\mathcal J|_{D})^n = 1$. As a consequence we have either the equality of cycles $E'_w = 2D$, or there exists an other irreducible component $D_2$ of $E'_w$ such that $E'_w = D + D_2$ and $(\mathcal J|_{D_2})^n = 1$. In this case $D_2 \simeq \bbP^n$ for the same reasons as above.
\end{enumerate}
\end{proof}

From \color{purple}Lemma \ref{lem:quadricbundle} \color{black}and \color{purple}Corollary \ref{cor:modifiedfibres}\color{black}, we can deduce information on the reducible fibres of $E_\mathrm{eq} \to Z_\mathrm{eq}$:

\begin{corollary}
\label{cor:originalfibres}
Let $z\in Z_\mathrm{eq}$ be a point such that the fibre $E_z = f^{-1}(z)\subset E_\mathrm{eq}$ is reducible. Then $E_z$ has two irreducible components, and the reduction of each component is normalized by $\bbP^n$.
\end{corollary}

\begin{proof}
Let $w$ be a point of $(\mu \circ \eta)^{-1}(z)$ and $E'_w$ the fibre over it, so that $E_z = \mu' \circ \eta'(E'_w)$. Since $E_z$ is reducible, so is $E'_w$ and by \color{purple}Corollary \ref{cor:modifiedfibres} \color{black}it is of the form $D_1 + D_2$ with $D_i \simeq \bbP^n$. In particular, $E_z$ has two components $G_1$ and $G_2$ with $G_i = \mu' \circ \eta'(D_i)$ for $i = 1,2$.

Through the isomorphism $D_i \simeq \bbP^n$ we have a morphism from $\bbP^n$ to the $n$-fold $G_i$
$$
\mu' \circ \eta' : D_i \simeq \bbP^n \to G_i.
$$
We know that a morphism from $\bbP^n$ to a variety of dimension $n$ does not contract any curve. Hence $D_i \to G_i$ is finite. Moreover, the restriction of the line bundle $\mathcal J$ to $D_i$ is the pullback of $L|_{G_i}$, so we have
$$
(L|_{G_i})^n = (\mathcal J|_{D_i})^n = 1,
$$
so $\bbP^n \to G_i$ is a birational and finite morphism. From this we deduce that $G_i$ is normalized by $\bbP^n$ for $i = 1,2$.
\end{proof}

\section{An example of a nonequidimensional divisorial elementary Mori contraction of submaximal length}

Consider in $\C^6$ with coordinates $(x_1,x_2,x_3,x_4,\lambda,\mu)$ the cubic affine cone
$$
Y = \left\{ \lambda x_1^2 + \lambda x_2^2 + \mu x_3^2 + \mu x_4^2 + x_1x_2^2 + x_1x_3^2 + x_1x_4^2 = 0\right\}.
$$
One can think of it as a family of affine cubics $Y_{(\lambda,\mu)} \subset \C^4_{(x_1,x_2,x_3,x_4)}$ indexed by $(\lambda,\mu) \in \C^2$. Let $\varepsilon : X \to Y$ be the blow-up of $Y$ along $\Lambda = \left\{ x_1 = x_2 = x_3 = x_4 = 0\right\}\simeq \C^2$.

\begin{lemma}
The variety $X$ is smooth.
\end{lemma}

\begin{proof}
Since $Y$ is smooth outside $\Lambda$, it is enough to prove that $X$ is smooth along $E = \varepsilon^{-1}(\Lambda)$. We have a model for $X$ inside $\bbP^3_{[u_1:u_2:u_3:u_4]}\times \C^6_{(x_1,x_2,x_3,x_4,\lambda,\mu)}$ with the following equations
$$
\left\{ \begin{array}{l}
\lambda u_1^2 + \lambda u_2^2 + \mu u_3^2 + \mu u_4^2 + x_2u_1u_2 + x_3u_1u_3 + x_4u_1u_4 = 0, \\
\mathrm{det}_2 \left( \begin{array}{cccc} u_1 & u_2 & u_3 & u_4 \\ x_1 & x_2 & x_3 & x_4 \end{array} \right) = 0.
\end{array} \right.
$$
Since the codimension of $X$ in $\bbP^3 \times \C^6$ is $4$, checking that $X$ is smooth amounts to showing that the Jacobian of the system of its equations has rank $4$ everywhere on $X$. Since $X$ is the blow-up of the affine cone $Y$ along an affine subspace which contains the singular points of $Y$, and $E$ is the exceptional divisor of the blow-up, $X$ is smooth on the complement of $E$. Hence we only need to check that the Jacobian of $X$ has maximal rank along $E$, in other words when we specify $x_i = 0$ for all $i$. Without loss of generality, we may assume $u_1 = 1$ and work with $(u_2,u_3,u_4)$ affine. A straightforward calculation shows that the partial Jacobian with respect to the variables $(x_1,x_2,x_3,x_4,\lambda,\mu)$ always has rank $4$.
\end{proof}

Consider now the blowdown morphism $\varepsilon : X \to Y$. Given a point $p = (\lambda,\mu)\in \Lambda - \left\{ (0,0)\right\}$, the fibre $E_{(\lambda,\mu)}$ over $p$ is the projectivization of the tangent cone of $Y_{(\lambda,\mu)}$ at the origin, in other words $E_{(\lambda,\mu)}$ is the quadric of $\bbP^3_{[u_1:u_2:u_3:u_4]}$ given by the equation
$$
\lambda u_1^2 + \lambda u_2^2 + \mu u_3^2 + \mu u_4^2 = 0.
$$
The fibre $E_{(0,0)}$ over the origin is the whole $\bbP^3$.

The $\varepsilon$-exceptional divisor $E$ is a hypersurface of $\bbP^3 \times \Lambda$ and it is generically a quadric bundle over $\Lambda$ with an isolated fibre which is isomorphic to $\mathbf P^3$. In $X$, the divisor $-E$ is $\varepsilon$-anti ample and the $(-E)$-minimal contracted curves are the lines contained in the fibres of $E \to \Lambda$.

\begin{lemma}
Let $\ell \subset E$ be a line contracted by $\varepsilon$. Then $-E \cdot \ell = 1$.
\end{lemma}

\begin{proof}
Let $\ell$ be a line contained in the general fibre $E_{(\lambda,\mu)}$. Then we may degenerate $\ell$ to a line $\ell_0$ contained in $E_{(0,0)} \simeq \bbP^3$, or contained in a singular quadric $E_{(\lambda',\mu')}$, yielding $-E\cdot \ell = -E\cdot \ell_0$. Hence we only need to check that $-E\cdot \ell = 1$ for $\ell$ contained in a smooth quadric $E_{(\lambda,\mu)} \simeq \bbP^1 \times \bbP^1$.

Consider the surface complete intersection 
$$
\Sigma = Y \cap \left\{ \lambda = \mu = 1, \, x_1 = 0 \right\}.
$$
This is the affine quadric cone in $\mathbf C^3$ given by the equation
$$
x_2^2 + x_3^2 + x_4^2 = 0.
$$
Let $\Sigma' \subset X$ be the proper transform of $\Sigma$ via $\varepsilon$, i.e., $\Sigma' \to \Sigma$ is the resolution of the singular point of $\Sigma$. This resolution contracts a smooth rational curve $\gamma = \Sigma' \cap E$ with $\gamma^2 = 2$. Moreover, $\gamma$ is by construction a hyperplane section of the quadric $E_{(1,1)}$. For any line $\ell$ in $E_{(1,1)}$ we have thus 
$$
-2E|_{E_{(1,1)}}\cdot \ell = -E|_{E_{(1,1)}}\cdot \gamma = -E|_{\Sigma'}\cdot \gamma = -\gamma^2 = 2,
$$
yielding $-E\cdot \ell = 1$ in $X$.
\end{proof}

\begin{lemma}
The blowdown morphism $\varepsilon : X \to Y$ is a divisorial elementary Mori contraction of submaximal length. 

The fibration $E \to \Lambda = \varepsilon(E)$, generically a quadric bundle, admits a fibre $E_0$ which is isomorphic to $\bbP^3$ with $-K_X|_{E_0}$ the hyperplane polarization.
\end{lemma}

\begin{proof}
The general fibre $E_p$ is embedded as a quadric surface in $\bbP^3$ and it satisfies
$$
\mathcal O_{E_p}(2) = -K_{E_p} = (-K_X-E)|_{E_p},
$$
and since $-E\cdot \ell = 1$ for $\ell$ a line contained in $E_p$ and $-K_{E_p}\cdot \ell = 2$, we have
$$
-K_X\cdot \ell = 1 = \dim E_p - 1
$$
for $\ell$ any line in the quadric $E_p$. 

If we degenerate the general line $\ell \in E_p$ to a line $\ell_0$ in the central fibre $E_0 = E_{(0,0)} \simeq \bbP^3$, we have $-K_X\cdot \ell_0 = 1$ and via the isomorphism $E_0\simeq \bbP^3$, the restriction $-K_X|_{E_0}$ is the polarization $\mathcal O_{\bbP^3}(1)$. In particular, the divisor $-K_X$ is relatively ample, hence $\varepsilon$ is a Mori contraction.

It remains to be proven that $\varepsilon$ is an elementary contraction. By the relative cone theorem (see for instance \cite[Theorem 7.51]{Deb01}) there exists a curve $C\subset E_0$ whose class is extremal in the relative Mori cone of $\varepsilon : X \to Y$, and an elementary contraction $\eta : X \to X'$ which contracts all the curves in the numerical equivalence class of $C$ and fits in the following commutative diagram:
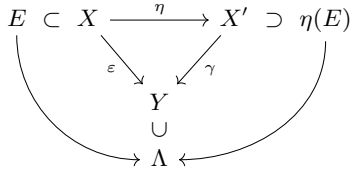
\begin{center}
\begin{tikzcd}[row sep=small,column sep=small]
E \arrow[r,phantom,"\subset"] \arrow[dddrr,out=270,in=180] & X \arrow[ddr,"\varepsilon",swap] \arrow[rr,"\eta"] & & X' \arrow[ddl,"\gamma"] & \eta(E) \arrow[l,phantom,"\supset"] \arrow[dddll,out=270,in=0]\\
& & & & \\
& & Y & & \\
& & \Lambda \arrow[u,phantom,sloped,"\subset"] & &
\end{tikzcd}
\end{center}
Since $E_0$ is isomorphic to $\bbP^3$, and $\eta$ contracts a curve $C\subset E_0$, then $\eta(E_0)$ is a point of $X'$. The restriction of $\gamma$ to $\eta(E)$ is thus a proper fibration over $\Lambda$ whose central fibre is a point; by semicontinuity $\gamma$ is locally an isomorphism over the origin of $Y$. This ensures that there exists a quadric fibre $E_p$ for $p\neq 0$ such that $\eta(E_p)$ is a point; as a consequence both families of lines on the general fibre of $E\to \Lambda$ are contracted, and as a consequence $\eta(E) \simeq \Lambda$ and $\gamma$ is an isomorphism. Therefore $\varepsilon$ is elementary. Moreover, it has submaximal length (\color{purple}Theorem \ref{thm:IonescuWisniewski}\color{black}) by the equality $l(\varepsilon) = -K_X\cdot \ell = 1 = \dim E_p -1$, for $\ell$ a line contained in the general fibre $E_p$.
\end{proof}

\vspace{1cm}

\small I\scriptsize NSTITUT DE \small M\scriptsize ATHÉMATIQUES DE \small T\scriptsize OULOUSE (CNRS UMR 5219), \small U\scriptsize NIVERSITÉ \small P\scriptsize AUL \small S\scriptsize ABATIER, 31062 \small T\scriptsize OULOUSE CEDEX 9, \small F\scriptsize RANCE

\vspace{.2cm}
\texttt{bruno.dewer@math.univ-toulouse.fr}


\begin{thebibliography}{1}

\bibitem[AW93]{AW} Marco Andreatta and Jaroslaw A. Wiśniewski, \textit{A note on nonvanishing and applications.} Duke Mathematical Journal \textbf{72}, No.~3, 1993.

\bibitem[BS95]{BS} Mauro C. Beltrametti and Andrew J. Sommese, \textit{The adjunction theory of complex projective varieties.} Volume 16 of de Gruyter Expositions in Mathematics. Walter de Gruyter \& Co., Berlin, 1995.

\bibitem[CMSB02]{CMSB} Koji Cho, Yoichi Miyaoka and Nicholas I. Shepherd-Barron, \textit{Characterizations of projective space and applications to complex symplectic manifolds.} Advanced Studies in Pure Mathematics \textbf{35}, p.~1--88. Mathematical Society of Japan, 2002.

\bibitem[Deb01]{Deb01} Olivier Debarre, \textit{Higher-dimensional algebraic geometry.} Universitext, Springer, 2001.

\bibitem[DH17]{DH} Thomas Dedieu and Andreas Höring, \textit{Numerical characterization of quadrics.} Algebraic geometry \textbf{4}, No.~1, p.~120--135, 2017.

\bibitem[FG65]{FG} Wolfgang Fischer and Hans Grauert, \textit{Lokal-triviale Familien kompakter komplexer Mannigfaltigkeiten.} Nachrichten der Akademie der Wissenschaften in Göttingen \textbf{2}, Vandenhoeck \& Ruprecht, 1965.

\bibitem[Har77]{Har} Robin Hartshorne, \textit{Algebraic Geometry.} Graduate Texts in Mathematics, Springer, 1977.

\bibitem[HN13]{HN} Andreas Höring and Carla Novelli, \textit{Mori contractions of maximal length.} Publ. RIMS \textbf{49}, No.~1, p.~215--228, 2013.

\bibitem[Hö12]{Hö} Andreas Höring, \textit{On a conjecture by Beltrametti and Sommese.} Journal of Algebraic Geometry \textbf{21}, p.~721--751, 2012.

\bibitem[Io86]{Io} Paltin Ionescu, \textit{Generalized adjunction and applications.} Mathematical Proceedings of the Cambridge Philosophical Society \textbf{99}, p.~457--472, 1986.

\bibitem[Ke02]{Ke} Stefan Kebekus, \textit{Characterizing the projective space after Cho, Miyaoka and Shepherd-Barron.} In \textit{Complex geometry (Göttingen, 2000)}, p.~147-–155. Springer, Berlin, 2002.

\bibitem[Kol96]{Kol} János Kollár, \textit{Rational curves on Algebraic Varieties.} Volume 32 of \textit{Ergebnisse der Mathematik und ihrer Grenzgebiete.} Springer-Verlag, 1996.

\bibitem[Kol11]{Kol3} János Kollár, \textit{Simultaneous normalization and algebra husks.} Asian Journal Of Mathematics \textbf{15}, No.~3, p.~437--450, 2011.

\bibitem[Kol13]{Kol2} János Kollár, \textit{Singularities of the minimal model program.} With the collaboration of Sándor Kovács. Cambridge University Press, 2013.

\bibitem[Mi04]{Mi} Yoichi Miyaoka, \textit{Numerical characterisations of hyperquadrics.} In \textit{Complex analysis in several variables --- Memorial conference of Kiyoshi Oka's centennal birthday.} Advanced Studies in Pure Mathematics \textbf{42}, p.~209--235, Mathematical Society of Japan, 2004.

\bibitem[Ro68]{Ro} Hugo Rossi, \textit{Picard variety of an isolated singular point}, Rice University Studies \textbf{54}, No.~4, p.~63–-73, 1968.

\bibitem[Wi91]{Wi} Jaroslaw A. Wiśniewski, \textit{On contractions of extremal rays of Fano manifolds.}  Journal für die reine und angewandte Mathematik \textbf{417}, p.~141--158, 1991.

\bibitem[Zar43]{Zar} Oscar Zariski, \textit{Foundations of a general theory of birational correspondences.} Transactions of the American Mathematical Society \textbf{53}, p.~490--542, 1943.

\end{thebibliography}
\end{document}